\documentclass[reqno]{amsart}

\usepackage{amsmath,amssymb,amsthm,mathtools}
\usepackage{mathrsfs}
\usepackage{enumitem}
\usepackage{hyperref}
\usepackage{xcolor}
\usepackage{graphicx}    
\usepackage{subcaption}
\numberwithin{equation}{section}

\newtheorem{thm}{Theorem}[section]
\newtheorem{prop}[thm]{Proposition}
\newtheorem{lem}[thm]{Lemma}
\newtheorem{cor}[thm]{Corollary}

\theoremstyle{definition}

\newcommand{\C}{\mathbb C}
\newcommand{\CP}{\mathbb{CP}}
\newcommand{\E}{\mathbb E}
\newcommand{\Prob}{\mathbb P}

\newcommand{\FS}{\mathrm{FS}}

\newcommand{\ind}{\mathbf 1}
\newcommand{\Zn}{\mathcal Z_n}

\newcommand{\beq}{\begin{equation}}
\newcommand{\eeq}{\end{equation}}
\newcommand{\mufs}{\mu_{\mathrm{FS}}}

\title[Nearest Zero--Critical Point Distances]
{Nearest zero--critical point distances for Gaussian SU(2) polynomials}

\author{Renjie Feng}
\email{renjie.feng@sydney.edu.au}
\address{Sydney Mathematical Research Institute, University of Sydney}
\date{\today}

\begin{document}

\begin{abstract}
We study nearest zero--critical point distances for the Gaussian SU(2)
polynomial \(p_n\). For each zero \(z_i\) of \(p_n\), let \(D_n(z_i)\) denote its
Fubini--Study distance to the nearest critical point. We prove that the
empirical measure of the rescaled distances \(nD_n(z_i)\) converges weakly
in probability to the law of \(2/|Z|\), where \(Z\) is distributed according
to the Fubini--Study probability measure in the affine chart.

\end{abstract}

\maketitle
 
\section{Introduction}
 
\subsection{Gaussian SU(2) polynomials}
The Gaussian SU(2) polynomial ensemble is one of the standard models in the
study of random analytic functions.  It is characterized by its
invariance under the natural SU(2) action on the Riemann sphere \(\CP^1\simeq S^2\).  Throughout the
paper we work in the affine chart
\(\CP^1\setminus\{\infty\}\simeq\C\).  In this chart the
 Gaussian SU(2) polynomial of degree \(n\) is
\[
        p_n(z)=\sum_{k=0}^n a_k\binom nk^{1/2}z^k, \quad a_k \sim N_{\C}(0,1),
\]
where the coefficients $\{a_k\}_{k=0}^n$ are independent and identically distributed (i.i.d.) standard complex Gaussian random variables, namely 
\[
        \E[a_k]=0,\qquad
        \E[a_k\overline{a_j}]=\delta_{kj},
        \qquad
        \E[a_k a_j]=0.
\]  The covariance kernel of \(p_n\) is
\begin{equation}\label{eq:kernel}
        K_n(z,w):=\E[p_n(z)\overline{p_n(w)}]
        =\sum_{k=0}^n\binom nk(z\overline w)^k
        =(1+z\overline w)^n.
\end{equation} 
Geometrically, \(p_n\) should be viewed as the local expression of a Gaussian
random holomorphic section of the line bundle
\[
        \mathcal O(n)=\mathcal O(1)^{\otimes n}\to\mathbb{CP}^1,
\]
where \(\mathcal O(1)\) is the hyperplane line bundle.  See \cite{BSZ,SZ} for background on Gaussian random holomorphic sections of line
bundles over K\"ahler manifolds.

The standard Fubini--Study metric on \(\CP^1\), written in the affine
coordinate, is 
\[
ds^2 = \frac{4|dz|^2}{(1+|z|^2)^2}.
\]
 Thus, for \(\Delta z\) small,  \beq\label{disdis}d_{\FS}(z, z+\Delta z) = \frac{2|\Delta z|}{1+|z|^2} + O(|\Delta z|^2).\eeq
The Fubini--Study distance on \(\CP^1\) agrees with
the spherical geodesic distance on the unit sphere \(S^2\). 
The associated Fubini--Study probability measure is
 \[d\mufs(z) := \frac{1}{\pi(1+|z|^2)^2} \, d\ell(z),\]
 where $d\ell(z) = dxdy$  denotes Lebesgue measure on \(\C\).  

Let
\[
        \mathcal Z_n=\{z_1,\dots,z_n\}
\]
be the zero set of \(p_n\).  The zeros are simple almost surely.  We define the
empirical zero measure by
\beq\label{mnns}
        \mu_n:=\frac1n\sum_{j=1}^n\delta_{z_j}.
\eeq
This is a random probability measure on \(\mathbb{CP}^1\). Its expectation is
understood in the weak sense: for every bounded measurable test function
\(\varphi\),
\[
        \mathbb E\left[\int \varphi\,d\mu_n\right]
        =
        \int \varphi\,d\mathbb E[\mu_n].
\]
The SU(2) invariance implies the exact identity
(e.g., \cite[Proposition~4.5]{SZ})
\[\label{identi}
        \mathbb E[\mu_n]=\mu_{\FS},
        \qquad n\ge1.
\]
Equivalently, the 1-point correlation function of zeros with respect to 
Lebesgue measure \(d\ell\) in the affine chart is
\begin{equation}\label{eq:1pt}
        \rho_1^{(n)}(z)
        =
        \frac{1}{\pi}\partial_z\partial_{\bar z}\log K_n(z,z)
        =
        \frac{n}{\pi(1+|z|^2)^2}.
\end{equation}
Indeed,
\[
        \frac1n\rho_1^{(n)}(z)\,d\ell(z)=d\mu_{\FS}(z).
\]
In addition to this identity in expectation, Shiffman--Zelditch further proved the weak convergence \cite[Proposition~4.2]{SZ},
\begin{equation}\label{wked}
        \mu_n
        \to 
        \mu_{\FS}
        \qquad\text{almost surely}.
\end{equation}
 That is, for every continuous
test function \(\varphi\in C(\mathbb{CP}^1)\),
\[
        \int \varphi\,d\mu_n
        \to
        \int \varphi\,d\mu_{\FS}
        \qquad\text{almost surely}.
\]
More generally, this almost sure convergence holds for zeros of Gaussian
random holomorphic sections of line bundles over K\"ahler manifolds
\cite[Theorem~1.1]{SZ}.

\subsection{Main results}
As mentioned above, Gaussian SU(2) polynomials should be viewed as
random holomorphic sections of \(\mathcal O(n)\to\mathbb{CP}^1\).  For a section
of a nontrivial line bundle, there is no canonical notion of derivative, and
hence no intrinsic notion of critical point, unless a connection is specified.  Different choices lead to different critical
point sets.

In this paper we fix a particular meromorphic connection
\(\nabla_{\mathrm{mero}}\) on \(\mathcal O(n)\).  This connection has a pole at a
point \(p'\in\mathbb{CP}^1\cong S^2\).  Let \(p\) be the antipodal point of
\(p'\) on \(S^2\).  We choose the affine coordinate \(z\) on
\[
        \mathbb{CP}^1\setminus\{p'\}\simeq\mathbb C
\]
so that
\[
        z(p)=0,
        \qquad
        z(p')=\infty .
\]
With respect to this affine coordinate, the connection we choose is represented locally by
\[
        \nabla_{\mathrm{mero}}^{(z)}=\frac{d}{dz}.
\]
Equivalently, in the coordinate \(w=1/z\) near the pole \(p'\), the same
connection is locally represented by
\[
        \nabla_{\mathrm{mero}}^{(w)}
        =
        \frac{d}{dw}
        -
        \frac{n}{w}.
\]
Thus the connection has a pole at \(w=0\), that is, at \(z=\infty\), or
equivalently at \(p'\).

The critical points considered in this paper are defined with respect to this
meromorphic connection.  Since the connection is represented by \(d/dz\) in the
chosen affine coordinate, the critical point set is
\[
        \mathcal C_n
        :=
        \{\beta\in\mathbb C:p_n'(\beta)=0\}.
\]
Note that the
critical point process \(\mathcal C_n\) is not SU(2)-invariant, even though
the zero process \(\mathcal Z_n\) is.

 As discussed heuristically in \S\ref{subsec:bulk-degenerate} below,  the zero--critical point
pairing has different asymptotic behaviour near the two exceptional points
\(z(p)=0\) and \(z(p')=\infty\).  Near \(p\), the Fubini--Study zero--critical point
distance is predicted to be larger than the \(n^{-1}\) bulk scale, whereas near \(p'\) it is predicted to be
smaller.  Thus both exceptional regimes differ from the bulk regime, where zeros are restricted to a compact subset of
\(\mathbb{CP}^1\setminus\{p,p'\}\). 
This behaviour reflects our choice of the meromorphic connection with a pole at \(p'\).

The study of the relation between zeros and critical points has a long history.  The
Gauss--Lucas theorem states that every critical point of a polynomial lies in
the convex hull of its zeros.  There is also a useful electrostatic
interpretation.  If
$ \{z_1, \dots, z_n\}$ are the zeros of a polynomial $q_n(z)$, then
\beq\label{criteq}
        \frac{q_n'(z)}{q_n(z)}
        =\sum_{j=1}^n\frac1{z-z_j}.
\eeq
Thus a critical point \(\beta\), satisfying \(q_n'(\beta)=0\), is an
equilibrium point of the logarithmic field generated by the zeros (see, e.g., \cite{DennisHannay2003}).  
 In this
paper we derive a scaling law for the distance from a zero to its nearest equilibrium point, namely its nearest critical point.

The phenomenon that zeros and critical points tend to form local pairs was
already observed in the physics literature.  Dennis--Hannay
\cite{DennisHannay2003} used a heuristic argument
to predict the zero--critical point spacing for certain random polynomials; see the discussion around equations (10)--(11) in
\cite{DennisHannay2003}.  Their argument also leads to predictions for the
global difference between the zero and critical point densities.

Rigorous local pairing theorems have been developed in a different, conditional
setting. A key input for the current paper is the local zero--critical point pairing theorem proved by Hanin in
\cite{hanin2, hanin3} (see Proposition~\ref{prop:Hanin-fixed} below).  
If \(\xi\in\CP^1\setminus\{0,\infty\}\) and the Gaussian SU(2) polynomial is
conditioned to have a zero at \(\xi\), then for every \(\varepsilon\in (0,1/2)\), with probability tending to one, there is a
unique critical point \(\beta\) satisfying $n^{-1-\varepsilon}<|\beta-\xi|<n^{-1+\varepsilon}.$

Similar conditional pairing phenomena have also been studied in several other
settings.  For example, one may consider random polynomials of the form
\beq\label{drtr}
        p_n^{\mathrm{i.i.d.}}(z)= (z-\xi)\prod_{i=1}^{n-1}(z-z_i),\eeq
where the \(z_i\)'s are i.i.d. random roots and \(\xi\) is a prescribed
 root.  In this setting, a local zero--critical point pairing
phenomenon emerges near the prescribed root \(\xi\)
\cite{Hanin2017,SW}.  Moreover, it was proved in \cite{KS} that a Gaussian
fluctuation theorem holds for the displacement between such a
conditioned zero and its nearby critical point in this i.i.d.-root model.

The present paper studies a different question.  We first sample a Gaussian
SU(2) polynomial and then, for each zero \(z_i\), measure the geodesic
distance on \(\mathbb{CP}^1\) to its nearest critical point.  We prove a
rescaled limiting law for the empirical measure of these distances.  Thus,
our result is not a conditional fixed-zero pairing theorem, but a
global empirical law for nearest zero--critical point distances.  This global
pairing phenomenon is illustrated in Figure~\ref{fig:combined2}.

\begin{figure}[htbp]
    \centering
    \includegraphics[width=0.45\textwidth]{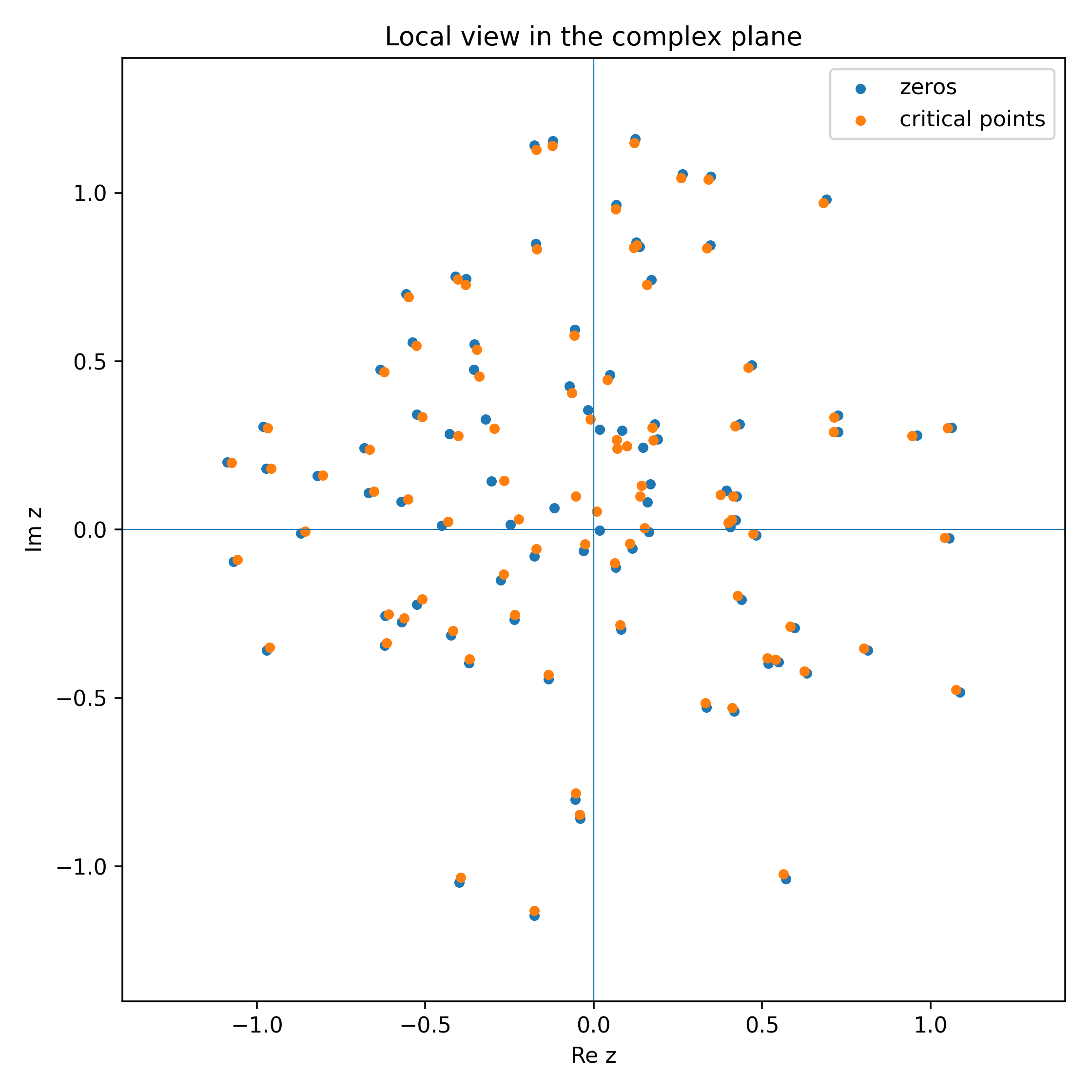}
    \caption{A sample of the zeros and critical points of a Gaussian SU(2)
    polynomial of degree \(n=150\).}
    \label{fig:combined2}
\end{figure}

To state the result precisely, define for each zero \(z_i\in\mathcal Z_n\) the Fubini--Study distance to the nearest critical point by
\[
        D_n(z_i):=\min_{\beta\in\mathcal C_n}d_{\FS}(z_i,\beta).
\]
We then define the empirical measure by
\begin{equation}\label{conmea}
        \nu_n:=\frac1n\sum_{i=1}^n\delta_{nD_n(z_i)}.
\end{equation} 
Let \(Z\sim\mu_{\FS}\), written in the affine coordinate, and define
\[\label{definey}
        Y:=\frac2{|Z|}.
\]
Then 
\beq\label{heavytail}
        \mathbb P(Y\le y)=\mathbb P\left(|Z|\geq \frac 2 y\right)=\frac{y^2}{y^2+4}, 
        \qquad y>0.
\eeq This distribution exhibits a heavy right tail.  Let \(\nu_Y\) denote the law of \(Y\).

\begin{thm}\label{main}
The random probability measures \(\nu_n\) converge weakly in probability to
\(\nu_Y\). Equivalently, for every bounded continuous function
\(f\in C_b([0,\infty))\),
\begin{equation}\label{eq:main-convergence-probability}
        \int f\,d\nu_n
        =
        \frac1n\sum_{i=1}^{n}f\bigl(nD_n(z_i)\bigr)
        \to
        \mathbb E[f(Y)]
        \quad\text{in probability.}
\end{equation}
\end{thm}

The proof of Theorem~\ref{main} is based on the following bulk estimate. 
Throughout the bulk analysis, we fix an arbitrary \(R>1\) and set
\begin{equation}\label{KR}
        K_R:=\{z\in\mathbb C:R^{-1}\le |z|\le R\}
        \subset \mathbb C\setminus\{0\}.
\end{equation}
All conditional probability statements involving \(K_R\) are understood with
this fixed \(R\).  Theorem \ref{main} will then be obtained in
\S\ref{mainproof} by letting \(R\to\infty\).

We define
\[
        V_{\mathrm{FS}}(z):=\frac{\bar z}{1+|z|^2}.
\] This field will appear as the macroscopic interaction field generated by the
other zeros (see Lemma \ref{lem:palm-field}). 

\begin{lem}
\label{lem:bulk-KR}
Let $R>1$ be fixed. Let \(I\) be uniformly distributed on \(\{1,\ldots,n\}\), independently of
\(p_n\), and let \(z_I\) be the corresponding zero.  Condition on the event
\(z_I\in K_R\).  Let
\(\beta_I\) be the nearest critical point to \(z_I\) in
Fubini--Study distance.  Then, in the affine coordinate, we have \begin{equation}\label{eq:bulk-displacement}
        n(\beta_I-z_I)
        +\frac 1{ V_{\mathrm{FS}}(z_I)}\to 0
\end{equation}
in probability, conditionally on \(z_I\in K_R\).  Consequently, the Fubini--Study distance satisfies 
\begin{equation}\label{eq:bulk-distance}
        nD_n(z_I)-\frac2{|z_I|}\to 0
\end{equation}
in probability, conditionally on \(z_I\in K_R\).  \end{lem}

\subsection{Bulk and degenerate regimes}
\label{subsec:bulk-degenerate}

We now explain heuristically why the limiting distribution in
\eqref{heavytail} has a heavy right tail, and why the proof naturally separates
the problem into a bulk regime and degenerate regimes.

The key object is the macroscopic interaction field generated by the other
zeros.  Let \(z_I\) be a uniformly chosen zero, and let \(\beta_I\) be its nearest
critical point.  In the proof of Lemma~\ref{lem:bulk-KR}, we will show that the leading-order
approximation is
\[
        \beta_I-z_I
        \approx
        -\frac1{nV_{\mathrm{FS}}(z_I)},
        \qquad
        D_n(z_I)\approx \frac2{n|z_I|}.
\]
This approximation is uniform only on regions where \(V_{\mathrm{FS}}\) is
bounded away from zero.  For this reason we first restrict to \(K_R\), defined
in \eqref{KR}.  On each fixed \(K_R\), the field \(V_{\mathrm{FS}}\) is bounded
away from zero, and the \(n^{-1}\)-scale zero--critical point pairing is stable.
This is the bulk regime.

There are two exceptional regimes where the macroscopic interaction
field degenerates or tends to zero. Note that 
\[
        V_{\mathrm{FS}}(0)=0,
        \qquad
        V_{\mathrm{FS}}(z)\to0
        \quad\text{as }\,\, |z|\to\infty .
\]
However, their effects on the Fubini--Study distance are very different.  The
bulk prediction gives
\[
        nD_n(z)\approx \frac2{|z|}\to0, \qquad |z|\to\infty. 
\]
Therefore, zeros near \(p'\), the pole of $\nabla_{\mathrm{mero}}$,  contribute small values to the
limiting law; equivalently, the zero--critical point distance  near $p'$ is predicted to be smaller than
the typical \(n^{-1}\) bulk scale.  

However, near the point \(p\), where \(z(p)=0\), 
\[
        nD_n(z)\approx \frac2{|z|}\to\infty,
        \qquad z\to0 .
\]
Therefore zeros close to \(p\) produce large values of the rescaled distance;
equivalently, the zero--critical point distance  near $p$ is predicted to be larger than the typical \(n^{-1}\)
bulk scale.  This mechanism leads to the heavy-tailed distribution in
\eqref{heavytail}.  This behaviour is also suggested by
Figure~\ref{fig:combined2}: near the origin, the zero--critical point pairing
appears less uniform, and larger zero--critical point distances are observed.

\subsection{Discussion and open problems}
\label{sec:discussion}
The present work suggests a general interaction field perspective for nearest
zero--critical point distance problems.  Suppose that a sequence of Gaussian
random polynomials admits a deterministic limiting zero distribution, and that the
field generated by the remaining zeros admits a deterministic macroscopic
limit as in Lemma \ref{lem:palm-field}.  Then one expects the local zero--critical point displacement to be
governed by this macroscopic interaction field, in the same way as
\(V_{\mathrm{FS}}\) governs the displacement in the Gaussian SU(2) case. 
However, the validity of this mechanism may depend on a separation of scales.  Roughly
speaking, the one-to-one pairing mechanism may persist only when
\beq\label{condddd}
        \text{zero--critical point spacing}
        \ll
        \text{minimum zero--zero spacing}.
\eeq
In the proof of Lemma~\ref{lem:bulk-KR}, this separation appears in
\eqref{ratio}: the ratio between the zero--critical point spacing and the
minimum  zero--zero spacing must tend to zero.  This
condition may fail in other classical models.  For example, for the classical Kac
polynomial
\[
        p_n^{\mathrm{Kac}}(z)
        =
        \sum_{k=0}^n a_k z^k,
        \qquad
        a_k\sim N_{\mathbb C}(0,1),
\]
the zeros accumulate near the unit circle, where the typical zero--zero spacing
is already of order \(n^{-1}\).  This is the same scale as the expected
zero--critical point spacing, which indicates that \eqref{condddd} fails.  Hence a critical point may interact with several
nearby zeros on the same scale, and the simple one-to-one pairing mechanism is
not expected to apply directly.

More generally, Theorem~1 of \cite{hanin2} establishes a fixed-zero pairing theorem
for general Gaussian random holomorphic sections over Riemann surfaces.  It is natural to ask whether an empirical nearest-distance
law can be proved in this broader setting.  Again, one expects a macroscopic
interaction field generated by the zeros to control the local displacement.
Understanding the degeneracy set of this field on the underlying Riemann surface
should be a key part of such a general theory.

One may also ask for fluctuation theorems.  A Gaussian fluctuation theorem was
proved in \cite{KS} for the displacement between the prescribed root
and its nearby critical point in the i.i.d.-root model; see \eqref{drtr}.  In
the Gaussian SU(2) setting, the analogous local problem is to understand the
fluctuations of 
\[
         n(\beta_I-z_I)+\frac1{V_{\mathrm{FS}}(z_I)} \quad\text{and}\quad   nD_n(z_I)- \frac2{|z_I|}
\]
for a uniformly chosen zero \(z_I\) in the bulk (away from both 0 and $\infty$).  This should be viewed as a
bulk-conditioned local fluctuation problem.

A global and more delicate problem is to study fluctuations of empirical
linear statistics such as
\[
        \frac1n\sum_{i=1}^n f\bigl(nD_n(z_i)\bigr)
\]
around their limiting means.  Because the limiting law has a heavy right tail,
the fluctuation behaviour is expected to depend sensitively on the class of test functions.  For compactly
supported smooth test functions, a Gaussian fluctuation theorem may still
be possible, but it would require quantitative control of the variance and the
correlations between the nearest zero--critical point distances associated with
different zeros.  We leave these questions for future work.

\section{Macroscopic interaction field}

Let $\label{zn}\Zn = \{z_1, \dots, z_n\}$
be the zero set of the Gaussian SU(2) polynomial \(p_n\).  For a zero \(z_i\), define
\[\label{eq:S-def}
        S_i:=\sum_{j\ne i}\frac1{z_i-z_j}.
\]
Recall that \(I\) is uniformly distributed on \(\{1,\dots,n\}\), independently of the polynomial.
\begin{lem}
\label{lem:palm-field}  
  Conditionally on \(z_I\in K_R\),
\begin{equation}\label{eq:palm-field}
        \frac1n S_I- V_{\mathrm{FS}}(z_I)\to 0
 \quad \mbox{in probability}. 
\end{equation}

\end{lem}

\begin{proof} Note that all zeros of $p_n$ are simple almost surely. 
At a zero \(z_i\) of \(p_n\),  \[\label{eq:log-derivative-identity}
        \frac{p_n''(z_i)}{p_n'(z_i)}
        =
        2\sum_{j\ne i}\frac1{z_i-z_j}
        =2S_i.
\]
Thus it suffices to prove, conditioning on \(z_I\in K_R\), 
\begin{equation}\label{eq:pprime-ratio-goal}
        \frac1{2n}\frac{p_n''(z_I)}{p_n'(z_I)}
       - \frac{\overline{z_I}}{1+|z_I|^2} \to 0
        \quad \mbox{in probability}.
\end{equation}
Fix \(z\in K_R\), and write
\[
        A=p_n(z),\qquad B=p_n'(z),\qquad C=p_n''(z),
        \qquad q=1+|z|^2.
\]
Next we analyze the asymptotic behavior of \(\frac1{2n}\frac{C}{B}\) under the conditional law given \(A=0\).
The vector \((A,B,C)\) is centered complex Gaussian. From the covariance kernel
\eqref{eq:kernel}, direct differentiation gives
$$\begin{aligned}
        \E[A\bar A]&=q^n,\label{eq:cov-AA}\\
        \E[B\bar A]&=n\bar z q^{n-1},\label{eq:cov-BA}\\
        \E[C\bar A]&=n(n-1)\bar z^2q^{n-2},\label{eq:cov-CA}\\
        \E[B\bar B]&=nq^{n-2}(1+n|z|^2),\label{eq:cov-BB}\\
        \E[C\bar B]&=2n(n-1)\bar z q^{n-2}
        +n(n-1)(n-2)z\bar z^2q^{n-3}.\label{eq:cov-CB}
\end{aligned}$$
Conditioning on \(A=0\), we obtain
\begin{equation}\label{eq:var-B-cond}
        \E[|B|^2\mid A=0]
        =
        \E[B\bar B]-\frac{|\E[B\bar A]|^2}{\E[|A|^2]}
        =
        nq^{n-2}.
\end{equation}
Similarly,
\begin{align*}
        \E[C\bar B\mid A=0]
        &=
        \E[C\bar B]-\frac{\E[C\bar A]\E[A\bar B]}{\E[|A|^2]}\notag\\
        &=
        2n(n-1)\bar z q^{n-3}.
        \label{eq:CB-cond}
\end{align*}
Therefore the conditional regression of \(C\) on \(B\), given \(A=0\), is
\begin{equation}\label{eq:C-regression}
        C=m_n(z)B+\Xi_n(z),
        \qquad
        m_n(z)=\frac{\E[C\bar B\mid A=0]}{\E[|B|^2\mid A=0]}
        =2(n-1)\frac{\bar z}{q},
\end{equation}
where \(\Xi_n(z)\) is independent of \(B\) under the conditional law
\(\Prob(\cdot\mid A=0)\).

We also need the conditional variance of the residual. Differentiating the
kernel gives
\[
\begin{aligned}
        \E[C\bar C]
        =
        n(n-1)q^{n-4}\bigl[&2q^2+4(n-2)|z|^2q 
        +(n-2)(n-3)|z|^4\bigr].
\end{aligned}
\]
A direct computation yields
\begin{align}
        \E[|C|^2\mid A=0]
        &=
        \E[C\bar C]-\frac{|\E[C\bar A]|^2}{\E[|A|^2]}
        \notag\\
        &=
        n(n-1)q^{n-4}\bigl[2+4(n-1)|z|^2\bigr].
        \label{eq:C-cond-var}
\end{align}
Using \eqref{eq:C-regression}, \eqref{eq:var-B-cond}, and
\eqref{eq:C-cond-var}, we find
\begin{align}
        \E[|\Xi_n(z)|^2\mid A=0]
        &=
        \E[|C|^2\mid A=0]
        -
        |m_n(z)|^2\E[|B|^2\mid A=0]
        \notag\\
        &=
        2n(n-1)q^{n-4}.
        \label{eq:residual-var}
\end{align} 

The Palm measure of the zero process at \(z\) is not simply the conditional law
\(\Prob(\cdot\mid A=0)\).  By the Kac--Rice formula, it is the
\(|B|^2\)-tilt of this conditional law:
\[
        d\mathbb Q_z
        :=
        \frac{|B|^2}{\E[|B|^2\mid A=0]}
        d\Prob(\cdot\mid A=0).
\]
We refer to \cite[Section~4.4]{cdk} for background on Palm measures.
In particular, for every measurable \(F=F(z,B,C)\) satisfying
\beq \label{condi2}
        \E\left[|B|^2|F(z,B,C)|\mid A=0\right]<\infty,
\eeq
we have
\beq\label{qzfz}
        \mathbb E_{\mathbb Q_z}[F(z,B,C)]
        =
        \frac{
        \mathbb E\left[|B|^2F(z,B,C)\mid A=0\right]
        }{
        \mathbb E\left[|B|^2\mid A=0\right]
        }.
\eeq

We now derive the Campbell--Palm identity for jet observables. In the present
setting, this identity is precisely the Kac--Rice formula rewritten in Palm
form. 

By the Kac--Rice formula, for any test function \(\psi\) (cf. \cite[Theorem 11.2.1]{AT}),
\[
        \mathbb E\left[\sum_{\zeta\in\mathcal Z_n}\psi(\zeta)\right]
        =
        \int_{\mathbb C}\psi(z)
        \mathbb E[|B|^2\mid A=0]\phi_A(0)\,d\ell(z),
\]
where $\phi_A(0)$ is the density at 0 of the complex Gaussian random variable $p_n(z)$. Thus the 1-point correlation function with respect to $d\ell$ is given by
\beq\label{rho1zd}
        \rho_1^{(n)}(z)
        =
        \mathbb E[|B|^2\mid A=0]\phi_A(0).
\eeq 
More generally, applying the Kac--Rice formula to the jet observable
\(F(z,B,C)\), we obtain
\[
\begin{aligned}
&\mathbb E\left[
        \sum_{\zeta\in\mathcal Z_n}
        \psi(\zeta)
        F\bigl(\zeta,p_n'(\zeta),p_n''(\zeta)\bigr)
        \right]  \\
&\qquad =
        \int_{\mathbb C}
        \psi(z)
        \mathbb E\left[
        |B|^2F(z,B,C)\mid A=0
        \right]\phi_A(0)\,d\ell(z).
\end{aligned}
\]
By \eqref{qzfz} and \eqref{rho1zd}, the preceding
Kac--Rice formula can be rewritten as the Campbell--Palm identity, 
\beq\label{campp}
\begin{aligned}
&\mathbb E\left[
        \sum_{\zeta\in\mathcal Z_n}
        \psi(\zeta)
        F\bigl(\zeta,p_n'(\zeta),p_n''(\zeta)\bigr)
        \right]  \\
&\qquad =
        \int_{\mathbb C}
        \psi(z)\rho_1^{(n)}(z)
        \mathbb E_{\mathbb Q_z}[F(z,B,C)]\,d\ell(z).
\end{aligned}
\eeq

We now finish the proof of Lemma \ref{lem:palm-field}. 
\(B\) is a non-degenerate complex Gaussian random variable under
\(\mathbb P(\,\cdot\,\mid A=0)\), and therefore
\[
        \mathbb P(B=0\mid A=0)=0.
\]
Since \(\mathbb Q_z\ll \mathbb P(\,\cdot\,\mid A=0)\), we also have
\[\mathbb Q_z(B=0)=0.\]  Hence the ratio \(C/B\) is well-defined
\(\mathbb Q_z\)-almost surely.  
Dividing \eqref{eq:C-regression} by
\(2nB\), we obtain, under \(\mathbb Q_z\),
\[\label{eq:CB-under-Q}
        \frac{1}{2n}\frac{C}{B}
        =
        \left(1-\frac1n\right)\frac{\bar z}{q}
        +
        \frac{1}{2n}\frac{\Xi_n(z)}{B}.
\] Note that the integrability condition \eqref{condi2} is satisfied for
\[
        \left|\frac1{2n}\frac{\Xi_n(z)}{B}\right|^2,
\]
since the factor \(|B|^2\) in the Palm measure cancels the
\(|B|^{-2}\) singularity.  By \eqref{eq:var-B-cond} and
\eqref{eq:residual-var}, there exists a uniform $C_R>0$ such that
\begin{align*}
        \E_{\mathbb Q_z}\left[
        \left|\frac1{2n}\frac{\Xi_n(z)}{B}\right|^2
        \right]
        &=
        \frac1{4n^2}
        \frac{\E[|\Xi_n(z)|^2\mid A=0]}{\E[|B|^2\mid A=0]}
        \notag\\
        &=
        \frac1{4n^2}
        \frac{2n(n-1)q^{n-4}}{nq^{n-2}}
        \notag\\
        &=
        \frac{n-1}{2n^2q^2}
        \le
        \frac{C_R}{n}.
        \label{eq:residual-tilt-bound}
\end{align*}
For any \(\eta>0\), since \(|\bar z|/q\) is bounded uniformly on \(K_R\), we
may choose \(n\) large enough so that
\[
        \frac1n\frac{|\bar z|}{q}<\frac{\eta}{2}
\]
uniformly for \(z\in K_R\). 
 By the decomposition
\[
\frac{1}{2n}\frac{C}{B} - \frac{\bar z}{q}
= -\frac{1}{n}\frac{\bar z}{q} + \frac{1}{2n}\frac{\Xi_n(z)}{B},
\]  and by the triangle inequality and Chebyshev's inequality, we have
\[ \mathbb Q_z\!\left( \Bigl|\frac{1}{2n}\frac{C}{B} - \frac{\bar z}{q}\Bigr| > \eta \right)
\le
\mathbb Q_z\!\left( \Bigl|\frac{1}{2n}\frac{\Xi_n(z)}{B}\Bigr| > \frac{\eta}{2} \right)
\le \frac{4}{\eta^2}\,
\mathbb E_{\mathbb Q_z}\!\left[ \Bigl|\frac{1}{2n}\frac{\Xi_n(z)}{B}\Bigr|^2 \right] \le \frac{4}{\eta^2}\frac{C_R}{n}.
\]
Therefore, 
 \begin{equation}\label{eq:Qz-bad-bound}
       \sup_{z\in K_R}
        \mathbb Q_z\left(
        \left|\frac1{2n}\frac{C}{B}-\frac{\bar z}{q}\right|>\eta
        \right)
        \to 0.
\end{equation}
It remains to pass from fixed \(z\) to a uniformly chosen zero. Let
\[
        \Phi_n(z)
        :=
        \ind\left\{
        \left|\frac1{2n}\frac{p_n''(z)}{p_n'(z)}
        -\frac{\bar z}{1+|z|^2}\right|>\eta
        \right\}.
\]
By the Campbell--Palm identity \eqref{campp}, 
\begin{align*}
        \E\left[\sum_{z_i\in\Zn\cap K_R}\Phi_n(z_i)\right]
        &=     \int_{K_R}
        \mathbb E_{\mathbb Q_z}[\Phi_n(z)]\rho_1^{(n)}(z)d\ell(z).
\end{align*}
Since  \(I\) is uniformly distributed on \(\{1,\dots,n\}\), independently of $p_n$, we have 
\begin{align*}
        \Prob\left(\Phi_n(z_I)=1\mid z_I\in K_R\right)
        &=
        \frac{
        \mathbb E\left[
        \frac1n\#\{z_i\in\Zn\cap K_R:\Phi_n(z_i)=1\}
        \right]
        }{
        \mathbb E\left[
        \frac1n\#\{z_i\in\Zn\cap K_R\}
        \right]
        }\\
        &=
        \frac{
        \int_{K_R}\mathbb E_{\mathbb Q_z}[\Phi_n(z)]\rho_1^{(n)}(z)\,d\ell(z)
        }{
        \int_{K_R}\rho_1^{(n)}(z)\,d\ell(z)
        }\\
        &\le
        \sup_{z\in K_R}\mathbb E_{\mathbb Q_z}[\Phi_n(z)].
\end{align*}
The last quantity tends to \(0\) by \eqref{eq:Qz-bad-bound}.  This proves
\eqref{eq:pprime-ratio-goal}, and hence \eqref{eq:palm-field}.
\end{proof}

\section{Zero--zero spacings}

In this section we collect some zero--zero spacing estimates needed later.  The
arguments use only the 1- and 2-point correlation functions of zeros.

Bleher--Shiffman--Zelditch proved that the rescaled correlation functions
of zeros of Gaussian random holomorphic sections over general K\"ahler
manifolds have universal limits \cite[Theorem~3.6]{BSZ}.  For Riemann surfaces,
further estimates were obtained in \cite{FengYao}. 
For the purposes of the present paper, we only need the following two upper bounds.   The first is the
global upper bound \cite[Theorem 5]{FengYao}, 
\begin{equation}\label{globds}
        \rho_2^{(n)}(z,w)\le Cn^2.
\end{equation}
The second is the short-range quadratic repulsion estimate (cf. \cite[Theorem 3]{FengYao}): for $n$ sufficiently large, uniformly for \(z\in K_R\) and
\(
        |z-w|\le \frac{\log^2 n}{\sqrt n},
\)
one has
\begin{equation}\label{localbound}
        \rho_2^{(n)}(z,w)
        \le
        C_R n^3 |z-w|^2 .
\end{equation}

Now we turn to the zero--zero spacing estimates for the Gaussian SU(2) polynomial. By the definition of correlation functions of zeros, if  \(I\) is uniformly distributed on \(\{1,\dots,n\}\), 
independently of $p_n$, then \[\label{eq:uniform-zero-one-point}
        \mathbb E\bigl[f(z_I)\bigr]
        =
        \frac1n
        \int_{\C}f(z)\rho^{(n)}_1(z)\,d\ell(z).
\]
Similarly,  
\begin{equation}\label{eq:uniform-zero-two-point}
\begin{aligned}
        \mathbb E\left[
        \sum_{j\ne I}g(z_I,z_j)
        \right]
        &=
        \frac1n
        \mathbb E\left[
        \sum_{i=1}^n\sum_{j\ne i}g(z_i,z_j)
        \right]  \\
        &=
        \frac1n
        \int_{\C}\int_{\C}
        g(z,w)\rho_2^{(n)}(z,w)\,d\ell(w)d\ell(z).
\end{aligned}
\end{equation}
In particular,
\begin{equation}\label{eq:KR-positive-mass}
        \mathbb P(z_I\in K_R)
        =\frac1n\int_{K_R}\rho^{(n)}_1(z)\,d\ell(z)=
        \int_{K_R}\frac{d\ell(z)}{\pi(1+|z|^2)^2}
        =:\mu_{\mathrm{FS}}(K_R). 
\end{equation}

 Note that the typical zero--zero spacing is of order \(n^{-1/2}\).  Moreover, the minimum zero--zero spacing among all
zeros is of order \(n^{-3/4}\), with a precise limiting distribution
\cite[Theorem~1]{FengYao}.   Motivated by this scale, we have the following
estimate for a uniformly chosen zero. 
\begin{lem} 
\label{lem:palm-spacing}
 For every \(\delta>0\),
\begin{equation}\label{eq:palm-spacing}
        \mathbb P\left(
        \min_{j\ne I}|z_I-z_j|\le n^{-3/4-\delta}
        \;\middle|\;
        z_I\in K_R
        \right)
        \to 0 .
\end{equation}
\end{lem}

\begin{proof}
Let \(r_n=n^{-3/4-\delta}\).  We first have the bound \[
\begin{aligned}
\mathbb P\left(
        z_I\in K_R,\,
        \min_{j\ne I}|z_I-z_j|\le r_n
        \right)  
 \le &
        \mathbb E\left[
        \mathbf 1_{\{z_I\in K_R\}}
        \sum_{j\ne I}\mathbf 1_{\{|z_I-z_j|\le r_n\}}
        \right]  \\
 =&
        \frac1n
        \int_{K_R}
        \int_{|w-z|\le r_n}
        \rho^{(n)}_2(z,w)\,d\ell(w)d\ell(z).
\end{aligned}
\]
 For \(n\) large enough, \eqref{localbound} gives 
\[
\begin{aligned}
        \frac1n
        \int_{K_R}
        \int_{|w-z|\le r_n}
        \rho^{(n)}_2(z,w)\,d\ell(w)d\ell(z)
        &\le
        \frac{C_R n^3}{n}
        \int_{K_R}
        \int_{|w-z|\le r_n}
        |z-w|^2\,d\ell(w)d\ell(z)  \\
        &\le
        C_R n^2 |K_R| r_n^4
        =
        C_R n^{-1-4\delta}.
\end{aligned}
\]
Dividing by \(\mathbb P(z_I\in K_R)=\mu_{\mathrm{FS}}(K_R)>0\) proves
\eqref{eq:palm-spacing}.
\end{proof}

For a sequence of random variables \(X_n\) and positive numbers \(a_n\), we
write \(X_n=O_{\mathbb P}(a_n)\) if, for every \(\varepsilon>0\), there exist
\(M>0\) and \(n_0\) such that
\[
        \mathbb P(|X_n|>Ma_n)<\varepsilon,
        \qquad n\ge n_0.
\] Similarly,   we write
\(
        X_n=o_{\mathbb P}(a_n)
\)
if, for every \(\eta>0\),
\[
        \mathbb P\bigl(|X_n|>\eta a_n\bigr)\to 0.
\]
We shall use the following elementary rules for \(O_{\mathbb P}\) and
\(o_{\mathbb P}\).  If
\(
        X_n=O_{\mathbb P}(a_n),
        Y_n=o_{\mathbb P}(b_n),
\)
then
\[
        X_nY_n=o_{\mathbb P}(a_nb_n).
\]
Moreover, if \(X_n=O_{\mathbb P}(a_n)\) and \(a_n/b_n\to0\), then
\[
        X_n=o_{\mathbb P}(b_n).
\]
When a conditioning event is specified, these statements are understood with
respect to the corresponding conditional probability measure.
 
\begin{lem}\label{lem:inv-square}
 Conditionally on \(z_I\in K_R\), \begin{equation}\label{eq:inv-square-bound}
        \sum_{j\ne I}\frac1{|z_I-z_j|^2}
        =
        O_{\mathbb P}(n\log n).
\end{equation}\end{lem}
As a related remark,  
\cite{Z}  studied
the Riesz energy of zeros on \(\CP^1\)
\[
        \mathcal E_s^n
        :=
        \sum_{i\ne j}
        \frac{1}{d_{\mathrm C}(z_i,z_j)^s},
\]
where \(d_{\mathrm C}\) is the chordal distance.   For \(s=2\),
\cite[Theorem~1.3]{Z} gives an asymptotic expansion, 
\[
        \E \mathcal E_2^n
        =
        \frac14 n^2\log n
        +
        \frac34 n^2\log\log n
        +
        O(n^2).
\]
 Other values of
\(s\) are also treated in \cite{Z}. The results in \cite{Z} were further generalized to higher-dimensional K\"ahler manifolds in \cite{FZ}.

Lemma~\ref{lem:inv-square} is a much rougher, but random, one-point analogue of
this energy estimate.  
Since the full two-point energy  $\mathcal E_2^n$ has order \(n^2\log n\), its average per zero
has order \(n\log n\), which is exactly the scale obtained in
Lemma~\ref{lem:inv-square}.  For our purposes this rough conditional estimate
is sufficient. For completeness, we include a brief proof.
\begin{proof}[Proof of Lemma \ref{lem:inv-square}]
  By
\eqref{eq:uniform-zero-two-point}, we have
\[
\begin{aligned}
&\mathbb E\left[
        \mathbf 1_{\{z_I\in K_R\}}
        \sum_{j\ne I}\frac1{|z_I-z_j|^2}
        \right]  =
        \frac1n
        \int_{K_R}
        \int_{\C}
        \frac1{|z-w|^2}\rho^{(n)}_2(z,w)\,d\ell(w)d\ell(z).
\end{aligned}
\]  We split the integral into three regions.  First, on \(|w-z|\le n^{-1/2}\), the short-range estimate \eqref{localbound} gives
\[
\begin{aligned}
    &   \int_{K_R}  \left[\int_{|w-z|\le n^{-1/2}}
        \frac1{|z-w|^2}\rho^{(n)}_2(z,w)\,d\ell(w)\right]d\ell(z)\\
        &\qquad \le 
        C_R n^3
      \int_{K_R}   \int_{|w-z|\le n^{-1/2}}d\ell(w)  d\ell(z)\le
        C_R n^2 .
\end{aligned}
\]
 Second, on \(n^{-1/2}<|w-z|\le c_R\) for some $c_R>0$ fixed, the global upper bound \eqref{globds}
 gives
\[
\begin{aligned}
      &   \int_{K_R}\left[ \int_{n^{-1/2}<|w-z|\le  c_R}
        \frac1{|z-w|^2}\rho^{(n)}_2(z,w)\,d\ell(w)\right]d\ell(z)\\
        &\qquad \le
        C_R n^2
        \int_{n^{-1/2}}^{c_R}\frac{dr}{r}  \le
        C_R n^2\log n.
\end{aligned}
\]
Third, on \(|w-z|> c_R\), the denominator is bounded below. 
The marginal
identity for the 2-point function is \[
        \int_{\C}\rho^{(n)}_2(z,w)\,d\ell(w)=(n-1)\rho^{(n)}_1(z).
\]
Therefore, 
\[
\begin{aligned}
&\int_{K_R}
        \int_{|w-z|>c_R}
        \frac{1}{|z-w|^2}
        \rho_2^{(n)}(z,w)\,d\ell(w)\,d\ell(z)  \\
&\qquad \le
        C_R
        \int_{K_R}
        \int_{|w-z|>c_R}
        \rho_2^{(n)}(z,w)\,d\ell(w)\,d\ell(z)  \\
&\qquad \le
        C_R
        \int_{K_R}
        (n-1)\rho_1^{(n)}(z)\,d\ell(z) \le
        C_R n^2 .
\end{aligned}
\]
Combining the three estimates and multiplying by the prefactor $1/n$, we obtain
\[
        \mathbb E\left[
        \mathbf 1_{\{z_I\in K_R\}}
        \sum_{j\ne I}\frac1{|z_I-z_j|^2}
        \right]
        \le
        C_R n\log n. 
\] 
Dividing by
\(
        \mathbb P(z_I\in K_R)=\mu_{\mathrm{FS}}(K_R)>0
\)
gives
\[
        \mathbb E\left[
        \sum_{j\ne I}\frac1{|z_I-z_j|^2}
        \;\Big |\;
        z_I\in K_R
        \right]
        \le
        C_R n\log n.
\]
 Given
\(\varepsilon>0\), choose \(M=C_R/\varepsilon\).  By Markov's inequality, 
\[
\begin{aligned}
&\mathbb P\left(
        \sum_{j\ne I}{|z_I-z_j|^{-2}}
        >
        M n\log n
        \;\Big |\;
        z_I\in K_R
        \right)  \\
&\qquad \le
        \frac{
        \mathbb E\left[
        \sum_{j\ne I}|z_I-z_j|^{-2}
        \mid z_I\in K_R
        \right]
        }{M n\log n}
        \le
        \frac{C_R}{M}
        =
        \varepsilon .
\end{aligned}
\]
This proves \eqref{eq:inv-square-bound}.
\end{proof}

\section{Zero--critical point spacings}
\label{sec:fixed-bulk}

We now prove Lemma \ref{lem:bulk-KR}.   The input is the following fixed-point conditional theorem.  For \(\xi\in K_R\) and $\varepsilon\in(0,1/2)$, define the event
\[
        \mathcal H_{n,\varepsilon}(\xi)
        :=
        \Big\{
        \#\bigl(\mathcal C_n\cap B(\xi,n^{-1+\varepsilon})\bigr)=1,\,
        \#\bigl(\mathcal C_n\cap B(\xi,n^{-1-\varepsilon})\bigr)=0
        \Big\}.
\]

\begin{prop}[{\cite[Theorem 1]{hanin2}}]
\label{prop:Hanin-fixed}
Let \(\xi\in K_R\) be fixed.  For every \(\varepsilon\in(0,1/2)\),
\[\label{eq:Hanin-fixed}
        \mathbb P\left(
        \mathcal H_{n,\varepsilon}(\xi)
        \,\middle|\,
        p_n(\xi)=0
        \right)
        \to 1 .\]
More precisely, there exists a constant \(C=C(\xi,\varepsilon)>0\) such that
\[\label{eq:Hanin-fixed-rate}
        \mathbb P\left(
        \mathcal H_{n,\varepsilon}(\xi)^c
        \,\middle|\,
        p_n(\xi)=0
        \right)
        \le
        C(\xi,\varepsilon)n^{-3/2+3\varepsilon}.
\]
\end{prop}
 Note that Proposition~\ref{prop:Hanin-fixed} is a special case of
\cite[Theorem~1]{hanin2}, which is proved for Gaussian random holomorphic sections
over Riemann surfaces.  We state only the form needed here for Gaussian
SU(2) polynomials.

 In the proof of Lemma~ \ref{lem:bulk-KR}, the zero is not fixed in advance; it is
chosen uniformly from the random zero set.   We therefore need to transfer the fixed-point
conditional statement Proposition \ref{prop:Hanin-fixed} to the law seen from a randomly selected zero.  

\begin{cor} 
\label{cor:Hanin-random-zero}
 For every \(\varepsilon\in(0,1/2)\), \[\label{eq:Hanin-random-zero}
\mathbb P\Bigl(  \mathcal H_{n,\varepsilon}(z_I)
        \,\mid\, z_I\in K_R
        \Bigr)
        \to 1.
\]
\end{cor}

\begin{proof}
For $\xi\in K_R$, write
\(
        A=p_n(\xi), B=p_n'(\xi) 
\). 
Recall that the Palm measure \(\mathbb Q_\xi\) of the zero process at
\(\xi\) is given by
\[
        d\mathbb Q_\xi
        =
        \frac{|B|^2}{\mathbb E[|B|^2\mid A=0]}
        \,d\mathbb P(\cdot\mid A=0).
\]
 The event \(\mathcal H_{n,\varepsilon}(\xi)\) is determined by the critical
point configuration of \(p_n\), 
\begin{align*}
\mathbf 1_{\mathcal H_{n,\varepsilon}(\xi)}& =
        \sum_{\beta\in\mathcal C_n}
        \mathbf 1_{\{\beta\in A_\xi\}}
        \prod_{\substack{u\in\mathcal C_n\\ u\ne\beta}}
        \left(1-\mathbf 1_{\{u\in U_\xi\}}\right),
\end{align*}
where \(
        U_\xi:=B(\xi,n^{-1+\varepsilon}), \,\, V_\xi:=B(\xi,n^{-1-\varepsilon}),\,\,
        A_\xi:=U_\xi\setminus V_\xi.
\)  Applying the Kac--Rice formula to the
corresponding indicator, and rewriting the result using the definition of
\(\mathbb Q_z\) as in the derivation of \eqref{campp}, gives the Campbell--Palm identity
\[
        \mathbb E\left[
        \sum_{z_i\in\mathcal Z_n}
        \mathbf 1_{\mathcal H_{n,\varepsilon}(z_i)}
        \right]
        =
        \int_{\mathbb C}
        \rho_1^{(n)}(z)\,
        \mathbb Q_z\bigl(\mathcal H_{n,\varepsilon}(z)\bigr)
        \,d\ell(z).
\] Here the Campbell--Palm identity is simply the Kac--Rice formula written under
the \(|p_n'(z)|^2\)-tilted conditional law \(p_n(z)=0\).

Now we show that, for each fixed \(\xi\in K_R\),
\begin{equation}\label{eq:pointwise-Palm-Hanin}
        \mathbb Q_\xi\bigl(\mathcal H_{n,\varepsilon}(\xi)^c\bigr)
        \to 0 .
\end{equation}
Indeed,  under \(\mathbb P(\cdot\mid A=0)\), the random variable \(B\) is a non-degenerate
centered complex Gaussian. Hence, by circular symmetry, 
\[
        \mathbb E[|B|^4\mid A=0]
        =2\mathbb E[|B|^2\mid A=0]^2 .
\]
By the Cauchy--Schwarz inequality,
\begin{align*}
        \mathbb Q_\xi\bigl(\mathcal H_{n,\varepsilon}(\xi)^c\bigr)
        &=
        \frac{
        \mathbb E\bigl[|B|^2\mathbf 1_{\mathcal H_{n,\varepsilon}(\xi)^c}
        \mid A=0\bigr]
        }{
        \mathbb E[|B|^2\mid A=0]
        }                                                       \notag\\
        &\le
        \left(
        \frac{\mathbb E[|B|^4\mid A=0]}
             {\mathbb E[|B|^2\mid A=0]^2}
        \right)^{1/2}
        \mathbb P\bigl(\mathcal H_{n,\varepsilon}(\xi)^c\mid A=0\bigr)^{1/2}
                                                                        \notag\\
        &\le
        \sqrt 2\, \mathbb P\bigl(\mathcal H_{n,\varepsilon}(\xi)^c\mid A=0\bigr)^{1/2}. 
        \label{eq:tilted-Hanin-bad}
\end{align*}
By Proposition~\ref{prop:Hanin-fixed}, the last quantity tends to \(0\) for
each fixed \(\xi\in K_R\).  This proves \eqref{eq:pointwise-Palm-Hanin}.

It remains to average over the location of the selected zero. Since \(I\) is chosen uniformly from \(\{1,\dots,n\}\), independently of \(p_n\), we have 
\begin{align*}
\mathbb P\bigl(\mathcal H_{n,\varepsilon}(z_I)^c\mid z_I\in K_R\bigr) &=
        \frac{
        \mathbb E\left[
        \sum_{z_i\in\mathcal Z_n\cap K_R}
        \mathbf 1_{\mathcal H_{n,\varepsilon}(z_i)^c}
        \right]
        }{
        \mathbb E\left[
        \#(\mathcal Z_n\cap K_R)
        \right]
        } .                                                       
\end{align*}
Applying the Campbell--Palm identity, we obtain \begin{align*}
\mathbb P\bigl(
        \mathcal H_{n,\varepsilon}(z_I)^c
        \,\mid\,
        z_I\in K_R
        \bigr) =& \frac{
        \int_{K_R}
        \mathbb Q_\xi\bigl(
        \mathcal H_{n,\varepsilon}(\xi)^c
        \bigr)
        \rho_1^{(n)}(\xi)\,d\ell(\xi)
        }{
        \int_{K_R}
        \rho_1^{(n)}(\xi)\,d\ell(\xi)
        }\\=&    \frac{
        \int_{K_R}
        \mathbb Q_\xi\bigl(
        \mathcal H_{n,\varepsilon}(\xi)^c
        \bigr)
        d\mu_{\mathrm{FS}}(\xi)
        }{
        \mu_{\mathrm{FS}}(K_R)
        } .\end{align*} For each fixed \(\xi\in K_R\), the integrand  $\mathbb Q_\xi\bigl(
        \mathcal H_{n,\varepsilon}(\xi)^c
        \bigr)$ tends to \(0\) by
\eqref{eq:pointwise-Palm-Hanin}, and it is bounded by \(1\).  Therefore, by the
dominated convergence theorem,
\[
        \int_{K_R}
        \mathbb Q_\xi\bigl(
        \mathcal H_{n,\varepsilon}(\xi)^c
        \bigr)
        d\mu_{\mathrm{FS}}(\xi)
        \to 0 .
\]
Since \(\mu_{\mathrm{FS}}(K_R)>0\),  it follows that
 \[
        \mathbb P\bigl(
        \mathcal H_{n,\varepsilon}(z_I)^c
        \,\mid \,
        z_I\in K_R
        \bigr)
        \to0.
\]
This proves the corollary.
\end{proof}

\begin{proof}[Proof of Lemma~\ref{lem:bulk-KR}]
Fix \(R>1\).  Throughout this proof we work under the conditional probability
measure
\[
        \mathbb P_R(\,\cdot\,)
        :=
        \mathbb P(\,\cdot\,\mid z_I\in K_R).
\]
All \(O_{\mathbb P}\) and \(o_{\mathbb P}\) statements below are understood with
respect to \(\mathbb P_R\).

Choose \(\varepsilon,\delta>0\) such that
\(
        \varepsilon+\delta<\frac14 .
\)
Let
\[a_{Ij}:=z_I-z_j,
        \qquad
        S_I:=\sum_{j\ne I}\frac1{a_{Ij}},
\]
and define the spacing event
\[
        E_n
        :=
        \left\{
        \min_{j\ne I}|a_{Ij}|\ge n^{-3/4-\delta}
        \right\}.
\]
Set
\[
        G_n
        :=
        \mathcal H_{n,\varepsilon}(z_I)\cap E_n .
\]
By Corollary~\ref{cor:Hanin-random-zero} and Lemma~\ref{lem:palm-spacing},
\[
        \mathbb P_R(G_n)\to1 .
\]
We shall carry out the estimates on \(G_n\).  Since
\(\mathbb P_R(G_n^c)\to0\), estimates proved on \(G_n\) imply the corresponding
\(O_{\mathbb P}\) and \(o_{\mathbb P}\) statements under \(\mathbb P_R\). For
example, to prove that \(X_n=o_{\mathbb P}(a_n)\) under \(\mathbb P_R\), it is
enough to show that, for every \(\eta>0\),
\[
        \mathbb P_R\bigl(|X_n|>\eta a_n,\;G_n\bigr)\to0.
\]
Indeed,
\[
        \mathbb P_R(|X_n|>\eta a_n)
        \le
        \mathbb P_R\bigl(|X_n|>\eta a_n,\;G_n\bigr)
        +
        \mathbb P_R(G_n^c),
\]
and both terms tend to zero.

On \(G_n\), let \[\beta_I=z_I+w_I\] be the unique critical point in
\(B(z_I,n^{-1+\varepsilon})\).  Then, by the definition of
\(\mathcal H_{n,\varepsilon}(z_I)\),
\begin{equation}\label{eq:Hanin-annulus-point}
        n^{-1-\varepsilon}<|w_I|<n^{-1+\varepsilon}.
\end{equation}
 At the critical point \(\beta_I=z_I+w_I\),  by \eqref{criteq}, we have 
\[\label{eq:crit-eq-main}
        0=
        \frac{p_n'(z_I+w_I)}{p_n(z_I+w_I)}
        =
        \frac1{w_I}
        +
        \sum_{j\ne I}\frac1{a_{Ij}+w_I}.
\]
Equivalently,
\begin{equation}\label{eq:crit-eq-expanded}
        0=\frac1{w_I}+S_I+\Delta_I,
\end{equation}
where
\[\label{eq:Delta-def}
        \Delta_I:=
        \sum_{j\ne I}
        \left(
        \frac1{a_{Ij}+w_I}-\frac1{a_{Ij}}
        \right).
\]
On \(G_n\), using \eqref{eq:Hanin-annulus-point} and the definition of \(E_n\),
we have
\begin{equation}\label{ratio}
        \max_{j\ne I}\left|\frac{w_I}{a_{Ij}}\right|
        \le
        n^{-1+\varepsilon}n^{3/4+\delta}
        =
        n^{-1/4+\varepsilon+\delta}
        =
        o(1).
\end{equation}
Hence, for all sufficiently large \(n\), on \(G_n\),
\[
        |a_{Ij}+w_I|
        \ge
        |a_{Ij}|\left(1-\frac{|w_I|}{|a_{Ij}|}\right)
        \ge
        \frac12|a_{Ij}|,
        \qquad j\ne I.
\]
Therefore, on \(G_n\),
\begin{align*}
        |\Delta_I|
        &\le
        \sum_{j\ne I}
        \frac{|w_I|}{|a_{Ij}+w_I|\,|a_{Ij}|}
        \le
        2|w_I|
        \sum_{j\ne I}\frac1{|a_{Ij}|^2}.
        \label{eq:Delta-basic-bound}
\end{align*}
Using \eqref{eq:Hanin-annulus-point} and Lemma~\ref{lem:inv-square}, which is
understood under \(\mathbb P_R\), we obtain
\begin{equation}\label{eq:Delta-opn}
        |\Delta_I|
        \le
        n^{-1+\varepsilon}O_{\mathbb P}(n\log n)
        =
        O_{\mathbb P}(n^\varepsilon\log n)
        =
        o_{\mathbb P}(n).
\end{equation}
On the other hand, Lemma~\ref{lem:palm-field}, again under the conditional law
\(\mathbb P_R\), gives
\begin{equation}\label{eq:S-asymptotic}
        S_I=nV_{\mathrm{FS}}(z_I)+o_{\mathbb P}(n).
\end{equation}
Since
       \[V_{\mathrm{FS}}(z)=\bar z/(1+|z|^2)\]
is bounded away from zero on \(K_R\), equations
\eqref{eq:Delta-opn} and \eqref{eq:S-asymptotic} imply
\[
        S_I+\Delta_I
        =
        nV_{\mathrm{FS}}(z_I)+o_{\mathbb P}(n)=  nV_{\mathrm{FS}}(z_I)\left(1+o_{\mathbb P}(1)\right).
\]
 From \eqref{eq:crit-eq-expanded}, we obtain
\beq\label{wid}
        w_I=-\frac1{S_I+\Delta_I}
        =
        -\frac1{nV_{\mathrm{FS}}(z_I)}+o_{\mathbb P}(n^{-1}),
\eeq
and thus 
\[\label{eq:bulk-displacement-proof}
        n(\beta_I-z_I)=nw_I
        =
        -\frac{1+|z_I|^2}{\overline{z_I}}
        +
        o_{\mathbb P}(1).
\]
This proves \eqref{eq:bulk-displacement}.

Finally, using the local Fubini--Study distance expansion (recall \eqref{disdis})
\[\label{eq:fs-local-distance-in-bulk}
        d_{\FS}(z,z+w)
        =
        \frac{2|w|}{1+|z|^2}+O_R(|w|^2),
        \qquad z\in K_R,
        \quad |w|\to0.
\]  Therefore, 
\begin{align*}
        nD_n(z_I)
        &=
        n\,d_{\FS}(z_I,z_I+w_I)                                      \\
        &=
        \frac{2n|w_I|}{1+|z_I|^2}+o_{\mathbb P}(1)                    \\
        &=
        \frac2{|z_I|}+o_{\mathbb P}(1).
\end{align*}
This proves \eqref{eq:bulk-distance}.
\end{proof}

\section{Proof of Theorem \ref{main}}\label{mainproof}

\begin{proof}
Fix \(f\in C_b([0,\infty))\).  Define
\[
        A_n(f)
        :=
        \frac1n\sum_{i=1}^{n}f\bigl(nD_n(z_i)\bigr).
\]
Let
\[
        g(z):=f\!\left(\frac2{|z|}\right),
\]
with the value at \(z=0\) chosen arbitrarily.  We also set \(g(\infty)=f(0)\).
Then \(g\) is a bounded measurable function on \(\mathbb{CP}^1\), and its set
of discontinuities is contained in \(\{0\}\).  Since
\(\mu_{\mathrm{FS}}(\{0\})=0\), \(g\) is \(\mu_{\mathrm{FS}}\)-almost everywhere
continuous.  Moreover,
\[
        \int g\,d\mu_{\mathrm{FS}}
        =
        \mathbb E[f(Y)].
\]
Set (recall \eqref{mnns} for the definition of $\mu_n$)
\[
        B_n(f)
        :=
        \frac1n\sum_{i=1}^{n}g(z_i)
        =
        \int g\,d\mu_n .
\]
Then
\[
        A_n(f)-\mathbb E[f(Y)]
        =
        \bigl(A_n(f)-B_n(f)\bigr)
        +
        \left(
        B_n(f)-\int g\,d\mu_{\mathrm{FS}}
        \right).
\]

We first prove that \(A_n(f)-B_n(f)\to0\) in probability.  Recall that \(I\) is
uniformly distributed on \(\{1,\dots,n\}\), independently of the polynomial.
Then
\[
\begin{aligned}
        \mathbb E|A_n(f)-B_n(f)|
        &\le
        \mathbb E\left[
        \frac1n\sum_{i=1}^{n}
        \left|
        f\bigl(nD_n(z_i)\bigr)
        -
        f\!\left(\frac2{|z_i|}\right)
        \right|
        \right]  \\
        &=
        \mathbb E\left[
        \left|
        f\bigl(nD_n(z_I)\bigr)
        -
        f\!\left(\frac2{|z_I|}\right)
        \right|
        \right].
\end{aligned}
\]
We claim that the last expectation tends to zero.

Fix $R>1$. By Lemma~\ref{lem:bulk-KR}, under the conditional law given
\(z_I\in K_R\),
\[
        nD_n(z_I)-\frac2{|z_I|}
        \to0
        \qquad\text{in probability}.
\]
Since \(f\) is continuous, it follows that, under the same conditional law,
\[
        f\bigl(nD_n(z_I)\bigr)
        -
        f\!\left(\frac2{|z_I|}\right)
        \to0
        \qquad\text{in probability}.
\]
Moreover,
\[
        \left|
        f\bigl(nD_n(z_I)\bigr)
        -
        f\!\left(\frac2{|z_I|}\right)
        \right|
        \le
        2\|f\|_{\infty}.
\]
Hence, by bounded convergence in probability, the above conditional convergence in probability implies conditional
\(L^1\)-convergence:
\begin{equation}\label{eq:conditional-L1-bulk}
        \mathbb E\left[
        \left|
        f\bigl(nD_n(z_I)\bigr)
        -
        f\!\left(\frac2{|z_I|}\right)
        \right|
        \,\middle|\,
        z_I\in K_R
        \right]
        \to0.
\end{equation}
Therefore,
\[
\begin{aligned}
&\mathbb E\left[
        \left|
        f\bigl(nD_n(z_I)\bigr)
        -
        f\!\left(\frac2{|z_I|}\right)
        \right|
        \right]  \\
&\quad \le
        \mathbb P(z_I\in K_R)
        \mathbb E\left[
        \left|
        f\bigl(nD_n(z_I)\bigr)
        -
        f\!\left(\frac2{|z_I|}\right)
        \right|
        \,\middle|\,
        z_I\in K_R
        \right]  \\
&\qquad
        +
        2\|f\|_{\infty}\,
        \mathbb P(z_I\notin K_R).
\end{aligned}
\]
The first term tends to zero for every fixed \(R>1\), by
\eqref{eq:conditional-L1-bulk}.  On the other hand, by the exact 1-point
intensity formula, recall \eqref{eq:KR-positive-mass},
\[
        \mathbb P(z_I\notin K_R)
        =
        \mu_{\mathrm{FS}}(K_R^c).
\]
Consequently,
\[
\limsup_{n\to\infty}
\mathbb E\left[
        \left|
        f\bigl(nD_n(z_I)\bigr)
        -
        f\!\left(\frac2{|z_I|}\right)
        \right|
        \right]
\le
        2\|f\|_{\infty}\mu_{\mathrm{FS}}(K_R^c).
\]
Since \(K_R\uparrow \mathbb C\setminus\{0\}\) and
\(\mu_{\mathrm{FS}}(\{0\})=0\), we have
\[
        \mu_{\mathrm{FS}}(K_R^c)\to0
        \qquad\text{as }R\to\infty.
\]
Letting \(R\to\infty\), we obtain
\[
        \mathbb E\left[
        \left|
        f\bigl(nD_n(z_I)\bigr)
        -
        f\!\left(\frac2{|z_I|}\right)
        \right|
        \right]
        \to0.
\]
Thus
\[
        \mathbb E|A_n(f)-B_n(f)|\to0.
\]
In particular,
\begin{equation}\label{adsds}
        A_n(f)-B_n(f)\to0
        \qquad\text{in probability}.
\end{equation}

It remains to identify the limit of \(B_n(f)\).  Recall from \eqref{wked} that the empirical zero
measure of Gaussian SU(2) polynomials satisfies the weak convergence 
\[
        \mu_n
        =
        \frac1n\sum_{i=1}^{n}\delta_{z_i}
        \to
        \mu_{\mathrm{FS}}
        \qquad\text{almost surely}.
\]
 Since
\(g\) is bounded and \(\mu_{\mathrm{FS}}\)-almost everywhere continuous, the
Portmanteau theorem gives
\begin{equation}\label{bconds}
        B_n(f)
        =
        \int g\,d\mu_n
        \to
        \int g\,d\mu_{\mathrm{FS}}
        =
        \mathbb E[f(Y)]
        \qquad\text{almost surely}.
\end{equation}
In particular, this convergence also holds in probability.

Combining \eqref{adsds} and \eqref{bconds}, we conclude that
\[
        A_n(f)
        \to
        \mathbb E[f(Y)]
        \qquad\text{in probability}.
\]
This proves \eqref{eq:main-convergence-probability}.
\end{proof}

\end{document}